\documentclass[10pt,a4paper]{article}

\usepackage[
  bookmarks=true,
  bookmarksopen=true,
  bookmarksnumbered=true,
  pdfauthor={Benjamin Miesch},
  pdftitle={Injective Metrics on Cube Complexes},
  pdfsubject={Injective Cube Complexes},
  pdfstartview=FitH,
	colorlinks=true,
	linkcolor=blue
]{hyperref}

\usepackage[english]{babel} 									

\usepackage{amsmath,amsfonts,amssymb}	
\usepackage{enumerate}
\usepackage{tikz}

\usepackage{amsthm}

\theoremstyle{plain}
\newtheorem{thm}{Theorem}[section]
\newtheorem{prop}[thm]{Proposition}
\newtheorem{cor}[thm]{Corollary}
\newtheorem{lem}[thm]{Lemma}

\theoremstyle{definition}
\newtheorem{definition}[thm]{Definition}
\newtheorem{example}[thm]{Example}

\theoremstyle{remark}
\newtheorem*{rmk}{Remark}
\newtheorem*{cl}{Claim}

\begin{document}

\title{\textsc{Injective Metrics on Cube Complexes}}
\author{Benjamin Miesch}
\date{\vspace{4ex}}
\maketitle

\begin{abstract}
	For locally finite CAT(0) cube complexes it is known that they are injectively metrizable choosing the $l_\infty$-norm on each cube. In this paper we show that cube complexes which are injective with respect to this metric are always CAT(0). Moreover we give a criterion for finite dimensional CAT(0) cube complexes with finite width to posses an injective metric. As a side result we prove a modification of Bridson's Theorem for cube complexes saying that finite dimensional cube complexes with $l_p$-norms on the cubes are geodesic.
\end{abstract}



\section{Introduction}


	A metric space $(X,d)$ is \emph{injective} if for every isometric embedding $\iota \colon A \hookrightarrow Y$ of metric spaces and every 1-lipschitz map $f \colon A \to X$ there is some 1-lipschitz map $\bar f \colon Y \to X$ such that $f = \bar f \circ \iota$. For instance complete real trees and $l_\infty^n$ are injective metric spaces.
	
	Furthermore cube complexes can be seen as a generalization of trees to higher dimensions. Especially CAT(0) cube complexes play an important role in recent research in geometric group theory as in Agol's proof of the Virtual Haken Conjecture \cite{agol}.
	
	The aim of this work is to investigate under which assumption a cube complex admits an injective metric. For finite median cube complexes it is known from work of van de Vel, Mai and Tang \cite{mai,vel} that they are collapsible and therefore admit an injective metric. But median cube complexes are the same as CAT(0) cube complexes as Chepoi showed in \cite{chepoi}. An easy argument then also shows that proper, i.e. locally finite, CAT(0) cube complexes are injectively metrizable.
	
	As we see there, choosing the $l_\infty$-norm on cubes provides a natural candidate for an injective metric on the cube complex. In section~\ref{sec:geodesic} we prove that for finite dimensional cube complexes this choice always results in a complete geodesic space. In fact we prove that this is true for any $l_p$-norm.  
	
\begin{thm}\label{thm:geodesic}
	Let $\mathcal{C}$ be a finite dimensional cube complex and let $d_p$ be the metric on $\mathcal{C}$ induced by the $l_p$-norm for $p \in [1,\infty]$. Then $(|\mathcal{C}|,d_p)$ is a complete geodesic space.
\end{thm}	
	
	This is a generalization of Bridson's Theorem \cite{bri} in the case of cube complexes.

	Gromov's link condition \cite{gromov} is a fundamental tool in proving the CAT(0) property of cube complexes. It says that a cube complex (with metric $d_2$) is CAT(0) if and only if it is simply-connected and all vertex links are flag. Using this we show in section~\ref{sec:injective} that $\mathcal{C}$ is CAT(0) whenever $(|\mathcal{C}|,d_\infty)$ is injective.
	
\begin{thm}\label{thm:CAT(0)}
	If the cube complex $(|\mathcal{C}|,d_\infty)$ is an injective metric space, then $(|\mathcal{C}|,d_2)$ is CAT(0).
\end{thm}

But the converse does not hold in general as we show in Example~\ref{ex:not injective}.
In section~\ref{sec:collapsible} we then investigate under which assumptions a CAT(0) cube complex is	injective. We follow the strategy of Mai and Tang \cite{mai} and extend their notion of cubical collapsibility. We say that a cube complex $\mathcal{C}$ is \emph{regularly collapsible} if there is a sequence of subcomplexes $\mathcal{C}_k$ for $k=0, \ldots, n$ such that $\mathcal{C}_0$ is a point, $\mathcal{C}_n=\mathcal{C}$ and each $\mathcal{C}_{k+1}$ arises from $\mathcal{C}_k$ by gluing a (possibly infinite) collection of cube complexes of the form ${L}_i \times [0,1]$ where the ${L}_i$'s are cuboids in $\mathcal{C}_k$. This gives us a new family of injective metric spaces which are not proper in general.

\begin{thm}\label{thm:injective}
	Let $\mathcal{C}$ be a regularly collapsible cube complex. Then $(|\mathcal{C}|,d_\infty)$ is an injective metric space.
\end{thm}

The aim is now to characterize all collapsible cube complexes. We do this by studying the hyperplanes of a CAT(0) cube complex $\mathcal{C}$. Each hyperplane separates $\mathcal{C}$ into two connected components, called \emph{halfspaces}. The\emph{ width} of $\mathcal{C}$ is given by the maximal number $n$ such that there is a chain $H_1 \subsetneq H_2 \subsetneq \ldots \subsetneq H_n$ of halfspaces. Furthermore the set of hyperplanes $\mathcal{H}$ is \emph{finitely colorable} if there is a function $f \colon \mathcal{H} \to F$ for some finite set $F$ such that $f(\mathfrak{h}) \neq f(\mathfrak{h}')$ whenever the hyperplanes $\mathfrak{h}$ and $\mathfrak{h}'$ intersect.

\begin{thm}\label{thm:collapsible}
	 A cube complex $\mathcal{C}$ is regular collapsible if and only if it is CAT(0), has finite width and its hyperplanes are finitely colorable.
\end{thm}

	As a consequence we get for all cube complexes $\mathcal{C}$ which are either locally finite or have finite width and finitely colorable hyperplanes that $(|\mathcal{C}|,d_\infty)$ is injective if and only if $(|\mathcal{C}|,d_2)$ is CAT(0).


\section{Preliminaries}



\subsection{Cube Complexes and Hyperplanes}


	A \emph{real $n$-cube} is the product $I^n = [0,1]^n$. A subspace of the form $F=\prod_{i=1}^n F_i$ with $F_i \in \{ \{0\}, \{1\}, [0,1] \}$ is called a \emph{face} of $I^n$.
	
	Let $X=\bigsqcup_{\lambda \in \Lambda} I_\lambda^{n_\lambda}$ be the disjoint union of cubes, $\sim$ an equivalence relation on $X$ and $p_\lambda \colon I_\lambda^{n_\lambda} \to X / \sim$ the projection maps. Assume that the following holds
	\begin{enumerate}[(i)]
	\item The map $p_\lambda \colon I_\lambda^{n_\lambda} \to X / \sim$ is injective for every $\lambda \in \Lambda$.
	\item If $p_\lambda (I_\lambda^{n_\lambda}) \cap p_\lambda (I_{\lambda'}^{n_{\lambda'}})\neq \emptyset$ then there are faces $F_\lambda \subset I_{\lambda}^{n_{\lambda}}$, $F_{\lambda'} \subset I_{\lambda'}^{n_{\lambda'}}$ and an isometry $h_{\lambda,\lambda'} \colon F_{\lambda} \to F_{\lambda'}$ such that for all $x \in I_\lambda^{n_\lambda},x' \in I_\lambda^{n_\lambda'} $ we have $p_\lambda(x)=p_{\lambda'}(x')$ if and only if $h_{\lambda,\lambda'}(x)=x'$.
	\end{enumerate}
	A subset $C \subset X/\sim$ we call a \emph{cube} if it is the image $p_\lambda (F)$ of some face $F \subset I_\lambda^{n_\lambda}$. The collection $\mathcal{C}= \{p_\lambda (F) : F \text{ face of } I_\lambda^{n_\lambda}, \lambda \in \Lambda \}$ of all cubes in $X/\sim$ is then a \emph{cube complex} and $|\mathcal{C}| = X/\sim$ its geometric realization. 
\\

Beside the standard euclidean metric $d_2$ we can also choose other metrics on the cubes $I^n$. We equip each cube $C$ with a length metric $d^C$ respecting the topology of $I^n \subset \mathbb{R}^n$ such that the maps $h_{\lambda,\lambda'}$ still are isometries, for instance 
\begin{align*}
	d_p(x,y) &= \|x-y\|_p=\sqrt[p]{\sum_{i=1}^n |x_i-y_i|^p}, \text{\quad for $p \geq 1$, or } \\
	d_\infty (x,y) &= \|x-y\|_\infty = \max \{|x_1-y_1|, \ldots , |x_n-y_n| \}.
\end{align*}
An \emph{($m$-)string} $\Sigma$ between two points $x,y \in |\mathcal{C}|$ is a tuple $(x_0, \ldots, x_m)$ of points in $\mathcal{C}$ such that $x=x_0, y=x_m$ and for any two consecutive points $x_i,x_{i+1}$ there is some cube $C_i \in \mathcal{C}$ containing $x_i,x_{i+1}$. For each $m$-string $\Sigma=(x_0, \ldots, x_m)$ we then get its \emph{length}
	\begin{equation}
		l(\Sigma)=\sum_{i=0}^{m-1} d^{C_i}(x_i,x_{i+1}).
	\end{equation}
	Using the length of strings we define the distance between any two points $x,y \in \mathcal{C}$ by taking the infimum over the length of all strings joining them.
	\begin{equation}
		d(x,y)=\inf \{ l(\Sigma) : \Sigma \text{ is a string from $x$ to $y$} \}.
	\end{equation}
If all metrics on the cubes are induced by the same $l_p$-norm we emphasize this by writing $d_p$. In the first case this is a pseudo-metric on the cube complex $\mathcal{C}$. But the proof of Theorem~I.7.13 in \cite{bridson} also works in the situation with any choice of length metrics on the cubes.

\begin{thm}\label{thm:complete length space}
	If $\mathcal{C}$ is a cube complex with length metrics on the cubes such that there are only finitely many isometry types, then $(|\mathcal{C}|,d)$ is a complete length space.
\end{thm}

For euclidean cubes Bridson's Theorem \cite{bri} applies and tells us that finite dimensional cube complexes with metric $d_2$ are geodesic. The more general case we treat in section~\ref{sec:geodesic}. For euclidean cube complexes Gromov \cite{gromov} gives a combinatorial criterion for the CAT(0) property. We state here a version used later and stated in \cite{chepoi}.

\begin{thm}
	A cube complex $\mathcal{C}$ is CAT(0) if and only if it is simply-connected and the following holds: Given three $(n+2)$-cubes intersecting in an $n$-cube and pairwise sharing an $(n+1)$-face then there is some $(n+3)$-cube containing all of them.
\end{thm}



	
In the study of CAT(0) cube complexes hyperplanes play a crucial role. A \emph{midcube} of an $n$-dimensional cube $I^n= [0,1]^n$ is a subset given by $M_k= \{x=(x_1, \ldots, x_n) \in I^n : x_k=\tfrac{1}{2} \}$.

	Let $\mathcal{C}$ be a CAT(0) cube complex and $E(\mathcal{C})$ the set of all edges. We define an equivalence relation on $E(\mathcal{C})$ generated by $e \Box f$ if $e$ and $f$ are opposite edges of some square. Given an equivalence class $[e]$ the \emph{hyperplane} dual to $[e]$ is the union of all midcubes in $\mathcal{C}$ crossing some edge in $[e]$.

We collect here some basic facts about hyperplanes.

\begin{thm}\label{thm:hyperplanes}\emph{(\cite{sageev}, Theorems 4.10 - 4.14)}
	For a CAT(0) cube complex $\mathcal{C}$ the following holds.
	\begin{enumerate}[(i)]
	\item Every hyperplane is embedded.
	\item Every hyperplane splits $|\mathcal{C}|$ into precisely two convex connected components.
	\item Every hyperplane is a CAT(0) cube complex.
	\item Every collection of pairwise intersecting hyperplanes intersects.
	\end{enumerate}
\end{thm}

Recall that a subset $A$ of a metric space $(X,d)$ is \emph{convex}, if for all points $x,y \in A$ the interval $I(x,y)=\{z : d(x,z)+d(z,y)=d(x,y)\}$ is contained in $A$.

\begin{definition}
	We call the set of hyperplanes $\mathcal{H}$ \emph{$n$-colorable} if there is a function $\Phi \colon \mathcal{H} \to \{ 1, \ldots n \}$ such that $\Phi(\mathfrak{h}) \neq \Phi (\mathfrak{h}')$ whenever the two hyperplanes $\mathfrak{h},\mathfrak{h}'$ intersect.
	$\mathcal{H}$ is \emph{finitely colorable} if it is $n$-colorable for some $n \in \mathbb{N}$.
\end{definition}

\begin{rmk}
	The task of finding a finite coloring of the hyperplanes can be translated to a graph coloring problem by defining the crossing graph $\Gamma(\mathcal{C})$ of $\mathcal{C}$ with vertices $\mathcal{H}$ where two vertices are connected if and only if the corresponding hyperplanes intersect. Then the set $\mathcal{H}$ is finitely colorable if and only if the chromatic number $\chi(\Gamma(\mathcal{C}))$ of $\Gamma(\mathcal{C})$ is finite. The chromatic number of the crossing graph for locally finite CAT(0) cube complexes has been studied by Chepoi and Hagen in \cite{chepoi-hagen}.
\end{rmk}

Since all midcubes in some given cube $C$ intersect and each of them belongs to a different hyperplane, finite colorability of the hyperplanes induces that the cube complex is finite dimensional. On the other hand, there are finite dimensional CAT(0) cube complexes with bounded diameter whose hyperplanes are not finitely colorable as  we will see later in Example~\ref{ex:not colorable}.


\subsection{Median Graphs}


	Let $G=(V,E)$ be a graph and $d$ the path metric on $G$ with edge length $1$.
	\begin{enumerate}[(i)]
	\item $G$ fulfills the \emph{triangle condition} if for all $u,v,w \in V$ with $1=d(v,w) < d(u,v)=d(u,w)$ there is some $x \in V$ which is a common neighbor of $v$ and $w$ such that $d(u,x)=d(u,v)-1$.
	\item $G$ fulfills the \emph{quadrangle condition} if for all $u,v,w,z \in V$ with $d(v,z)=d(w,z)=1$ and $1\leq d(u,v)=d(u,w)=d(u,z)-1$ there is some $x \in V$ which is a common neighbor of $v$ and $w$ such that $d(u,x)=d(u,v)-1$.
	\end{enumerate}

	The graph $G$ is called \emph{median} if $(V,d)$ is a median metric space, i.e. for any three points $x,y,z \in V$ there is a unique point $\mu(x,y,z) \in I(x,y)\cap I(y,z)\cap I(z,x)$. For further details and results on median graphs see \cite{chepoi} and the references given therein.

\begin{rmk}
	Median graphs are triangle-free and fulfill both the triangle and the quadrangle condition. Since such graphs are triangle-free, the triangle condition implies that there are no two neighbors having the same distance to a third vertex.
\end{rmk}

\newpage
\begin{example}
	The following graph, denoted by $K_{2,3}$, is triangle-free, fulfills both the triangle and the quadrangle condition but is not median.
\begin{center}
\begin{tikzpicture}
  [scale=.3,auto=left,every node/.style={circle,fill=black!50, inner sep=0pt, minimum width=4pt}]
  \node (n1) at (1,5)  {};
  \node (n2) at (5,2)  {};
  \node (n3) at (5,5)  {};
  \node (n4) at (5,8)  {};
  \node (n5) at (9,5)  {};

  \foreach \from/\to in {n1/n2,n1/n3,n1/n4,n2/n5,n3/n5,n4/n5}
    \draw (\from) -- (\to);
\end{tikzpicture}
\end{center}
\end{example}

\begin{prop} \emph{(\cite{chepoi}, Lemma 4.1)}
	A graph $G$ is median if and only if it is triangle-free, fulfills the quadrangle condition and does not contain a copy of $K_{2,3}$.
\end{prop}

	Let $G=(V,E)$ be a graph. A \emph{subgraph} $A$ of $G$ is a graph with vertex set $V(A) \subset V$ and containing all edges $xy \in E$ with $x,y \in V(A)$.
	Given two subgraphs $A,B$ we define the union $A\cup B$ and the difference $A \setminus B$ as the subgraphs with vertex set $V(A\cup B) = V(A) \cup V(B)$ and $V(A\setminus B) = V(A) \setminus V(B)$ respectively.
	
	A subset $A$ of a metric space $(X,d)$ is \emph{gated} if for all $x \in X$ there is some $x_0 \in A$ such that for all $a\in A$ we have $d(x,a)=d(x,x_0)+d(x_0,a)$. Clearly if such an $x_0$ exists it is unique and we then call $x_0$ the \emph{gate} of $x$ in $A$.

\begin{lem}\label{lem:convex subgraphs}\cite{chep}
	Let $H$ be a connected subgraph of the median graph $G$. Then $H$ is convex if and only if it is $2$-convex, i.e. for every pair $x,y \in H$ with $d(x,y)=2$ we have $I(x,y) \subset H$. Moreover every convex subset of $G$ is gated.
\end{lem}

\begin{proof}
	Convex sets are obviously $2$-convex. Therefore we must only prove the converse.	
	Let $x,y \in H$ be any two points in the $2$-convex, connected subgraph $H$ of $G$ with distance $d(x,y)=n$. We show by induction on $n$ that $I(x,y) \subset H$. For $n=0,1,2$ this is clearly true since $H$ is $2$-convex. Thus assume that $I(x,y) \subset H$ for any $x,y \in H$ with $d(x,y)\leq n$. Now let $x,y \in H$ be any two points with $d(x,y)=n+1$ and $z\in I(x,y)$. We then find a discrete geodesic $x=a_0, a_1, \ldots, a_{n+1}=y$ with $a_k=z$. Moreover let $y=b_0, b_1, \ldots, b_l=x$ be a shortest discrete path from $y$ to $x$ in $H$. Since $G$ is median we have for all $u,v \in G$ with $d(u,v)=1$ that $d(x,u)=d(x,v) \pm 1$. Choose $j \in \{1,\ldots,l\}$ minimal such that $d(x,b_{j+1})=d(x,b_j)-1$. Then $d(x,b_{j-1})=d(x,b_j)-1$ and we find by the quadrangle condition some point $b_j'  \in G$ with $d(x,b_j')=d(x,b_{j-1})-1$ and $d(b_{j-1},b_j')=1$. By $2$-convexity $b_j'\in H$ and thus $b_j'\neq b_{j-2}$ since otherwise $b_0, \ldots, b_{j-2}, b_j, \ldots, b_l$ would be a shorter discrete path from $y$ to $x$ in $H$. Inductively we find $b_i'\in H$ with $d(b_{i-1},b_i')=1$ and  $d(x,b_i')=d(x,b_{i-1})-1$ for $i=0, \ldots, j$. We apply again the quadrangle condition for $x,a_n,y,b_1'$ and get some $w \in G$ with $d(x,w)=n-1$. By the induction hypothesis $w\in H$ since $w \in I(x,b_1')$. Therefore $a_n \in H$ since $H$ is $2$-convex and by using the induction hypothesis a second time we finally have $z \in H$.
	
	It remains to show that convex subsets are gated. First we claim that for any $x \in G$ there is a unique $x_0 \in H$ with $d(x,x_0)=d(x,H)$. Indeed if there are two points $x_0,x_0'$ with this property let $x_0=b_0, \ldots, b_l=x_0'$ be a discrete geodesic. As above we find some $b_1'\in H$ with $d(x,b_1')=d(x,x_0)-1 < d(x,H)$. A contradiction.
	
	For any $x \in G$ take $x_0 \in H$ realizing $d(x,H)=d(x,x_0)$. We show that $x_0$ is the gate of $x$. Let $z \in H$ and choose a discrete geodesic $x_0=b_0, \ldots b_l=z$. If $d(x,z) < d(x,x_0) + d(x_0,z)$ we can again find some $b_1' \in H$ with $d(x,b_1)=d(x,x_0)-1$ as before.
\end{proof}

\begin{lem}\label{lem:gates of gates}\cite{vel_convex}
	Let $X$ be a metric space and $A,A' \subset X$ two gated subsets with gate functions $p \colon X \to A$ and $p' \colon X \to A'$. Then we have $p' \circ p \circ p' = p'$ on $A$.
\end{lem}

\begin{proof}
	First observe that $a$ is the gate of $x$ in $A$ if and only if $I(a,x) \cap A =\{a\}$. Let $x$ be any point in $A$ and $a'$ its gate in $B$. Moreover denote the gate of $a'$ in $A$ by $a$. We have $I(x,a') \cap A' = \{a'\}$ and $I(a',a)\cap A =\{a\}$. Since $a$ is the gate of $a'$ in $A$ it follows that $a \in I(x,a')$ and therefore $I(a,a') \subset I(x,a')$ and $I(a,a') \cap A' =\{a'\}$. Hence $a'$ is also the gate of $a$ in $A'$.
\end{proof}

	Let $\mathcal{C}$ be a cube complex. Then its $1$-skeleton defines a graph, called the \emph{underlying graph} of $\mathcal{C}$.

\begin{rmk}
	Given a connected subgraph $A$ of the underlying graph it naturally defines a subcomplex containing all cubes with vertices in $V(A)$. Therefore we can always consider the subgraph as a subcomplex and vice versa. 
\end{rmk}

\begin{thm} \emph{(\cite{chepoi}, Theorems 6.1 \& 6.8)}
	Let $\mathcal{C}$ be a cube complex. Then the following statements are equivalent:
	\begin{enumerate}[(i)]
	\item $(|\mathcal{C}|,d_2)$ is CAT(0).
	\item Its underlying graph $G$ is a median graph.
	\item $(|\mathcal{C}|,d_1)$ is a median metric space.
	\end{enumerate}
\end{thm}

In general a subset $H$ of a metric space $X$ is called a \emph{halfspace} if $H$ and $X\setminus H$ are convex. But as a consequence of the following lemma if we speak of a halfspace in a CAT(0) cube complex we always refer to a halfspace arising from a splitting by some hyperplane. Note that this convention also presumes that halfspaces of CAT(0) cube complexes are non-empty and strict subsets.

\begin{lem}\label{lem:halfspaces and hyperplanes}
	Let $G$ be a median graph arising as the underlying graph of a CAT(0) cube complex. Then every hyperplane splits $G$ into two halfspaces and conversely every pair of non-empty complementary halfspaces is separated by a hyperplane.
\end{lem}

\begin{proof}
	The first implication is a direct consequence of Theorem~\ref{thm:hyperplanes}.
	Conversely let $H,H'$ be two complementary halfspaces. By Lemma~\ref{lem:convex subgraphs} there are gate maps $p \colon G \to H, p' \colon G \to H'$. Take $x \in p(H')$ and $x'=p'(x)$. Then $p(x')=x$ (Lemma~\ref{lem:gates of gates}) and $d(x,x')=1$. Therefore the halfspaces are given by
	\begin{align*}
		H &=G(x,x')=\{y \in G : d(x,y)< d(x',y) \} \text{ and } \\
		H'&=G(x',x)=\{y \in G : d(x',y)< d(x,y) \}
	\end{align*}
	and according to Lemma~6.4 in \cite{chepoi} these two sets are separated by the hyperplane transversal to the edge $xx'$.
\end{proof}

\begin{definition}
	The \emph{width} of a CAT(0) cube complex $\mathcal{C}$ is the maximal length of a chain of halfspaces, i.e.
	\begin{equation}
		\operatorname{width}(\mathcal{C})= \sup \{n : \exists H_1 , \ldots , H_n \text{ halfspaces with } H_1 \subsetneq H_2 \subsetneq \ldots \subsetneq H_n \}.
	\end{equation}
\end{definition}

	If $\mathfrak{h}$ is a hyperplane that splits the CAT(0) cube complex $\mathcal{C}$ into two halfspaces with one of them minimal (and hence the other one is maximal) then we call $\mathfrak{h}$ \emph{extremal}.

\begin{lem}\label{lem:convex image}
	Let $G$ be a median graph with two complementary halfspaces $H$ and $H'$. Then the image of $H'$ under the gate function $p \colon G \to H$ is a convex set.
\end{lem}

\begin{proof}
	Using Lemma~\ref{lem:convex subgraphs} it is enough to show $2$-convexity. Let $x,y \in p(H')$ with $d(x,y)=2$ and assume that there is some $z_0 \in I(x,y) \setminus p(H')$. Clearly $z_0$ is a neighbor of $x$ and $y$. By Lemma~\ref{lem:gates of gates} we have that $x=p(p'(x))$ and $y=p(p'(y))$. Denote $x'=p'(x), y'=p'(y)$ and observe that $d(x,x')=1$ and $d(y,y')=1$. Moreover we have $d(x',y')=2$ since gate functions are distance non-increasing. There is at least one common neighbor $z'$ of $x'$ and $y'$ which is mapped to a common neighbor $z$ of $x$ and $y$ such that $z$ and $z'$ mutually are gates of each other. We then have three squares $zxx'z'$, $zyy'z'$ and $zxz_0y$ which pairwise intersect in one edge and have a common vertex $z$. Therefore the link condition for CAT(0) cube complexes implies that this squares must be contained in a $3$-cube. Especially there is some further corner $z_0'$ of this cube which lies in $H'$ and is a neighbor of $z_0$. It follows that $z_0=p(z_0') \in p(H')$ which contradicts our assumption.
\end{proof}


\section{Geodesic Cube Complexes}\label{sec:geodesic}


By the Hopf-Rinow Theorem (see Theorem I.3.7 in \cite{bridson}) we get as a direct consequence of Theorem~\ref{thm:complete length space} that locally compact, i.e. locally finite, cube complexes with finitely many isometry types of cubes are geodesic. Here we prove the statement in a more general setting.

\begin{thm}\label{thm:geodesic cube complexes}
	Let $\mathcal{C}$ be a cube complex with finitely many isometry types of cubes where the metrics on the cubes $[0,1]^n$ are induced by norms on $\mathbb{R}^n$ satisfying
	\begin{equation}\label{eq: good norms}
		\|\bar x \| \leq \|x \| \text{ if } |\bar x^l | \leq |x^l| \text{ for all } l=1,\ldots, n.
	\end{equation}
	Then $(|\mathcal{C}|,d)$ is a complete geodesic space.
\end{thm}

Since (\ref{eq: good norms}) is true for all $l_p$-norms, Theorem~\ref{thm:geodesic} follows.
\\

We first prove the statement for $d_\infty$ and then discuss in a second step which properties of $\| \cdot \|_\infty$ are really needed. For the first step we follow the strategy of the proof of Theorem I.7.19 in \cite{bridson} which is based on finding taut strings that approximate distances. Since geodesics in $(\mathbb R^n , d_\infty)$ are not unique we must demand further properties for taut strings.

In the following let $\mathcal{C}$ be a cube complex endowed with the metric $d_\infty$. Given some string $\Sigma = (x_0, \ldots, x_m)$ let $C_i$ be the largest cube containing $x_i,x_{i+1}$. Furthermore let $F_i=C_i \cap C_{i-1}$ be the common face of the two consecutive cubes. Introduce inductively coordinates $(x^1, \ldots , x^n)$ on each cube $C_i=\prod_{l=1}^n I_l \subset [0,1]^n$ with $I_l=\{0\}$ or $[0,1]$ such that the following holds:
\begin{enumerate}[(i)]
	\item $x_0^l \leq x_1^l$ for all $l$,
	\item $(0, \ldots ,0) \in F_i$ and
	\item on $F_i$ the coordinates of $C_{i-1}$ and $C_i$ coincide (modulo $1$).
\end{enumerate}

\begin{definition}\label{def:taut}
	An $m$-string $\Sigma = (x_0, \ldots , x_m)$ in $\mathcal{C}$ is \emph{taut} if it satisfies the following conditions for $i=1, \ldots, m-1$:
	\begin{enumerate}[(i)]
	\item There is no cube containing $\{x_{i-1},x_i,x_{i+1}\}$.
	\item If $x_{i-1},x_i \in C_{i-1}$ and $x_i, x_{i+1} \in C_i$ then $x_i \in I_{L_i}(x_{i-1},x_{i+1})$ in $L_i = C_{i-1} \cup C_i$, where $I_{L_i}$ denotes the metric interval with respect to the induced length metric on $L_i$.
	\item For each $l$ we have $x_i^l \leq x_{i+1}^l$.
	\end{enumerate}
\end{definition}

\begin{rmk}
	In the euclidean case, property (ii) already implies that the coordinates $x_i^l$ are monotone, but for $d_\infty$ we have to demand (iii) additionally to ensure that taut strings which cross many cubes are long.
\end{rmk}

\begin{lem}\label{lem:taut strings are long}
	Let $\mathcal{C}$ be a finite dimensional cube complex with dimension bounded by $n$. If $\Sigma = (x_0, \ldots , x_{n+2})$ is a taut $(n+2)$-string then $l(\Sigma) \geq 1$.
\end{lem}

\begin{proof}
	By property (iii) we get that $x_i^l \leq x_{i+1}^l$ for all $i,l$.	
	For $i=1, \ldots, n+1$ we then make the following observations. Since $x_i \in F_i$ but $x_{i+1} \notin F_i$ there is some $l_i$ such that $x_i^{l_i}=0$ and $x_{i+1}^{l_i} \neq 0$.
	This leads to chains $0=x_i^{l_i} < x_{i+1}^{l_i} \leq \ldots$ for $i=1, \ldots, n+1$. Since $l_i \in \{1, \ldots, n\}$ by the pigeonhole principle there are $i_1< i_2$ with $l_{i_1}=l_{i_2}$. But this implies that the chain $0=x_{i_1}^{l_{i_1}} < x_{i_1+1}^{l_{i_1}} \leq \ldots$ ends with some $x_{i'}^{l_{i_1}}=1$, $i'< i_2$ and therefore 
	\begin{align*}
	l(\Sigma) \geq \sum_{i=i_1+1}^{i'} d_\infty(x_{i-1},x_{i}) \geq \sum_{i=i_1+1}^{i'} |x_{i}^{l_{i_1}}-x_{i-1}^{l_{i_1}}| = x_{i'}^{l_{i_1}}-x_{i_1}^{l_{i_1}}=1.
	\end{align*}
\end{proof}

\begin{cor}\label{cor:taut strings are long}
	Let $\mathcal{C}$ be a finite dimensional cube complex. Then there is some constant $\alpha$ such that for every taut $m$-string $\Sigma=(x_0, \ldots, x_m)$ we have $l(\Sigma) \geq \alpha m - 1$.
\end{cor}

\begin{lem}\label{lem:tauten minimal m-strings}
	For some fixed $m$, let $\Sigma=(x_0, \ldots , x_m)$ be an $m$-string from $x$ to $y$ of minimal length among all $m$-strings joining these points. Then there is some taut n-string $\bar\Sigma=(\bar x_0, \ldots , \bar x_m)$ from $x$ to $y$ with $l(\bar \Sigma ) = l(\Sigma)$ and $n\leq m$.
\end{lem}

\begin{proof}
	We proceed by induction on $m$. A $1$-string clearly is taut. Therefore lets assume $m\geq 2$. If there are three consecutive points $x_{i-1},x_i,x_{i+1}$ all contained in some cube $C$ we can simply cancel $x_i$ and get an $(m-1)$-string $\bar \Sigma =(x_0, \ldots , x_{i-1},x_{i+1},\ldots, x_m)$ with the same length by the triangle inequality and minimality. Hence we may assume that property (i) of taut strings holds. If property (ii) does not hold we can easily find some $\bar{x}_i \in F_i$ with $d(x_{i-1},\bar x_i)+d(\bar x_i,x_{i+1}) < d(x_{i-1}, x_i)+d(x_i,x_{i+1})$ contradicting minimality of $\Sigma$. Thus it remains to achieve that $\bar x_{i-1}^l \leq \bar x_i^l \leq \bar x_{i+1}^l$. By induction we may assume that this is true for $i \leq m-2$. For each $l$ with $x_{m-1}^l > x_m^l$ consider the maximal chain $0 \leq x_{i_0}^l \leq \ldots \leq x_{m-1}^l$. If $x_{i_0}^l \leq x_m^l$ define $\bar x_i^l = \min \{ x_i^l, x_m^l \}$ for $i=i_0, \ldots, m-1$. Otherwise we have $i_0=0$, i.e. $x_{i_0}=x$. If so define $\bar x_i^l = x^l$ for $i=1, \ldots, m-1$ and reflect the corresponding coordinate axes at $\frac{1}{2}$. In either case we have $|\bar x_i^l - \bar x_{i+1}^l| \leq |x_i^l - x_{i+1}^l|$ and hence $\bar \Sigma$ is no longer than $\Sigma$. Observe that if $x_m^l=0$ it may happen that property (i) does not hold anymore for $\bar \Sigma$ but then we can proceed as described at the beginning of the proof.
\end{proof}

\begin{lem}\label{lem:shortest m-string in finite complex}
	Let $\mathcal{L}$ be a finite cube complex and $x,y \in |\mathcal{L}|$ two points that can be joined by an $m$-string for some fixed $m$. Then there is a shortest $m$-string from $x$ to $y$.
\end{lem}

\begin{proof}
	Adapt the proof of Lemma I.7.25 in \cite{bridson}.
\end{proof}

\begin{lem}\label{lem:shortst m-string in finite dimensional complex}
	Let $\mathcal{C}$ be a finite dimensional cube complex and $x,y\in |\mathcal{C}|$ two points that can be joined by an $m$-string for some fixed $m$. Then there is a shortest $m$-string joining them.
\end{lem}

\begin{proof}
	Since the dimension of the cubes is bounded, there are only finitely many types of cubes in $\mathcal{C}$. Hence there are also only finitely many isometry types of subcomplexes consisting of $m$ cubes. Now every $m$-string is contained in such a subcomplex and for every type of subcomplex there is some smallest $m$-string joining $x$ and $y$ by the previous lemma. Therefore we can find a smallest $m$-string between $x$ and $y$ in $|\mathcal{C}|$.
\end{proof}

\begin{prop}\label{prop:metric by taut stings}
	Let $\mathcal{C}$ be a finite dimensional cube complex. Then for any two points $x$ and $y$ in $|\mathcal{C}|$  we have
	\begin{equation}
		d_\infty(x,y)=\inf \{l(\Sigma) : \Sigma \text{ is a taut string from $x$ to $y$} \}.
	\end{equation}		
\end{prop}

\begin{proof}
	By Lemma~\ref{lem:shortst m-string in finite dimensional complex} and Lemma~\ref{lem:tauten minimal m-strings} for every $m$-string $\Sigma$ there is some taut $n$-string $\bar \Sigma$ with the same endpoints and $l(\bar \Sigma) \leq l(\Sigma)$.
\end{proof}

\begin{thm}\label{thm:finite dimensional cube complexes are geodesic}
	Let $\mathcal{C}$ be a finite dimensional cube complex. Then $(|\mathcal{C}|,d_\infty)$ is a complete geodesic space.
\end{thm}

\begin{proof}
	Let $x,y$ be two points in $|\mathcal{C}|$ with distance $d:=d_\infty(x,y)$. Corollary~\ref{cor:taut strings are long} then tells us that there is some $m_0$ such that we can approximate $d_\infty(x,y)$ only using $m$-stings with $m \leq m_0$. But then by Lemma~\ref{lem:shortst m-string in finite dimensional complex} there is some shortest string $\Sigma=(x_0. \ldots,x_m)$ joining $x$ and $y$, i.e. $d_\infty(x,y) = \sum_{i=1}^n d_\infty(x_{i-1},x_i)$. Connecting the $x_i$'s by straight lines, we finally get a geodesic from $x$ to $y$.
\end{proof}

\begin{proof}[Proof of Theorem~\ref{thm:geodesic cube complexes}.]
	If we reread the proof we see that it also works for metrics induced by some other norms on the cubes. The crucial points are that
	\begin{enumerate}[(i)]
	\item there are minimal $m$-strings among all $m$-strings (Lemmas~\ref{lem:shortst m-string in finite dimensional complex}),
	\item we can approximate minimal $m$-strings by taut strings (Lemma~\ref{lem:tauten minimal m-strings}) and
	\item taut $m$-strings are long (Lemma~\ref{lem:taut strings are long}).
	\end{enumerate}
	
	Requirement (i) is given as long as there are only finitely many isometry types of cubes, i.e. only finitely many different norms occurring. For (ii) we need that the norms fulfill 
	\begin{equation*}
		\|\bar x \| \leq \|x \| \text{ if } |\bar x^l | \leq |x^l| \text{ for all } l=1,\ldots, n.
	\end{equation*} 
	This is used in the last step of the proof. Property (iii) easily follows assuming (i) since norms on finite dimensional vector spaces are equivalent.
\end{proof}


\section{Injective Cube Complexes are CAT(0)}\label{sec:injective}


\begin{definition}
	A metric space $(X,d)$ is \emph{($m$-)hyperconvex} if for any collection $\{ \bar{B}(x_i,r_i)\}_{i\in I}$ of closed balls with $d(x_i,x_j) \leq r_i + r_j$ (and $|I| \leq m$) we have 
	\begin{equation*}
		\bigcap_{i\in I} \bar{B}(x_i,r_i) \neq \emptyset.
	\end{equation*}		
\end{definition}

This terminology was introduced by Aronszajn and Panitchpakdi \cite{aronszajn} who proved the following result.

\begin{prop}
	A metric space is injective if and only if it is hyperconvex.
\end{prop}

\begin{lem}\label{lem1}
	Let $\mathcal{C}$ be a cube complex endowed with the maximum metric $d_\infty$ such that $(|\mathcal{C}|,d_\infty)$ is geodesic. Assume that there are cubes $C,C_0 \in \mathcal{C}$ with $C \subset C_0$, a point $p \in C_0$ and $r \in [0, \tfrac12]$ with $d_\infty(p,\partial C) \geq r$ and $d_\infty(p,C) \leq r$. Then for $q \in \bar{B}(p,r)$ there is some cube $D$ containing $q$ and $C$.
\end{lem}

\begin{proof}
	Let $\Sigma=(p_0 , p_1 , \ldots , p_k)$ be a taut string from $p$ to $q$ and $C_0, C_1, \ldots , C_k$ maximal cubes with $p_i,p_{i+1} \in C_i$. Observe that for $i\neq 0,k$ we have $p_i \in \partial C_i$.
	
	We denote $n= \dim C$, $n_i= \dim C_i$ and identify $C$ with $[0,1]^n \times \{0\}^{n_i-n} \subset [0,1]^{n_i}=C_i$ if $C \subset C_i$. By our assumption $p=(x_1, \ldots, x_{n_0})_{C_0}$ with $r \leq x_l \leq 1-r$ for $l=1, \ldots, n$ and $x_l \leq r$ for $l=n+1, \ldots, n_0$.
	\begin{cl}
		For $i=0,\ldots, k$ we have $C \subset C_i$ and $p_i=(x_1, \ldots, x_{n_i})_{C_i}$ with $r-d_\infty(p,p_i) \leq x_l \leq 1-r+d_\infty(p,p_i)$ for $l=1,\ldots, n$ and $x_l \leq r + d_\infty(p,p_i)$ for $l=n+1, \ldots n_i$.
	\end{cl} 
For $i=0$ this clearly is true. For $p_{i+1}=(y_1, \ldots , y_{n_{i}})_{C_{i}}$ we can easily check the inequalities using $|x_l-y_l | \leq d_\infty(p_i,p_{i+1})$ and $d_\infty(p,p_{i+1})=d_\infty(p,p_i)+d_\infty(p_i,p_{i+1})$. Since $p_{i+1} \neq q$ we have $d_\infty(p,p_{i+1})=d_\infty(p,q)-d_\infty(q,p_{i+1}) < r$ and therefore $y_l \neq 0,1$ for $l=1, \ldots , n$ and $y_l \neq 1$ for $l=n+1, \ldots , n_i$. Thus the subcube $D_{i+1}$ of $C_i\cap C_{i+1}$ with $p_{i+1} \in \operatorname{Int}{C_{i+1}'}$ also contains $C$.
\end{proof}

\begin{thm}
	Let $\mathcal{C}$ be a cubical complex such that $(|\mathcal{C}|,d_\infty)$ is geodesic, simply connected and locally $3$-hyperconvex. Then $(|\mathcal{C}|,d_2)$ is CAT(0).
\end{thm}

\begin{proof}
	We must check the link condition. Let $A,B,C$ be three $(k+2)$-cubes such that $A \cap B \cap C$ is a $k$-cube and $A \cap B$, $A \cap C$ and $B \cap C$ are $(k+1)$-cubes. We need to show that $A$, $B$ and $C$ are contained in some $(k+3)$-cube. Assume that they are not.
	
	Note that then there is also no cube containing $C$ and $A\cap B$. Indeed if they are contained in such a cube $D$, then $A \cap C$ and $A \cap B$ are two different $(k+1)$-faces of the same $(k+2)$-face of $D$ which therefore coincides with $A$. Similarly $B \subset D$.
	
	The goal is now to find three closed balls in $\mathcal{C}$ that intersect pairwise but have no common intersection point. For doing this, we introduce coordinates on our cubes $A$,$B$ and $C$ such that 
\begin{align*}
	A \cap B &= \{(x_1, \ldots , x_{k+2})_A : x_{k+2}=0 \} &= \{(y_1, \ldots , y_{k+2})_B : y_{k+2}=0 \} \\
	A \cap C &= \{(x_1, \ldots , x_{k+2})_A : x_{k+1}=0 \} &= \{(z_1, \ldots , z_{k+2})_C : z_{k+2}=0 \} \\
	B \cap C &= \{(y_1, \ldots , y_{k+2})_B : y_{k+1}=0 \} &= \{(z_1, \ldots , z_{k+2})_C : z_{k+1}=0 \}
\end{align*}	
	
	Choose $\epsilon \in (0,1)$ such that $\bar{B}((\tfrac{1}{2},\ldots, \tfrac{1}{2}, 0,0)_A,\epsilon)$ has the $(3,2)$-IP. Define $x=(\tfrac{1}{2},\ldots, \tfrac{1}{2}, \tfrac{\epsilon}{8},0)_A = (\tfrac{1}{2},\ldots, \tfrac{1}{2}, \tfrac{\epsilon}{8},0)_B$, $y=(\tfrac{1}{2},\ldots, \tfrac{1}{2}, \tfrac{3\epsilon}{8},\tfrac{\epsilon}{8})_C$ and $z=(\tfrac{1}{2},\ldots, \tfrac{1}{2}, \tfrac{\epsilon}{8},\tfrac{3\epsilon}{8})_C$. It is easy to check that $d_\infty(x,y),d_\infty(x,z) \leq \tfrac{\epsilon}{4} + \tfrac{\epsilon}{8}$ and $d_\infty(y,z) \leq \tfrac{\epsilon}{8} +\tfrac{\epsilon}{8}$. By the $(3,2)$-IP we therefore have $\bar{B}(x,\tfrac{\epsilon}{4}) \cap \bar{B}(y,\tfrac{\epsilon}{8}) \cap \bar{B}(z,\tfrac{\epsilon}{8}) \neq \emptyset$. Let $p$ be any point in this intersection. Since $p\in \bar{B}(y,\tfrac{\epsilon}{8}) \cap \bar{B}(z,\tfrac{\epsilon}{8})$ we have $d(p, \partial C) \geq \tfrac{2 \epsilon}{8}$ and $d(p,c)\leq \tfrac{\epsilon}{8}$. Hence by Lemma~\ref{lem1} there is some cube $D$ containing $p$ and $C$. Thus $p=(u_1, \ldots ,u_k, \tfrac{\epsilon}{4},\tfrac{\epsilon}{4},u_{k+3}, \ldots, u_n)_D$ with $u_i \in [\tfrac{1}{2}-\tfrac{\epsilon}{8},\tfrac{1}{2}+\tfrac{\epsilon}{8}]$ for $i=1, \ldots,k$ and $u_i \in [0,\tfrac{\epsilon}{8}]$ for $i=k+3, \ldots,n$. Especially $p$ satisfies the assumption of Lemma~\ref{lem1} for $r=\tfrac{\epsilon}{4}$ and since there is no cube containing $x$ and $C$ we have $d_\infty(p,x) > \tfrac{\epsilon}{4}$ which contradicts $p \in \bar{B}(x,\tfrac{\epsilon}{4})$.
\end{proof}

Since injective metric spaces are geodesic, simply connected and ($3$-)hyperconvex, we conclude that Theorem~\ref{thm:CAT(0)} holds. But the converse is false in general as the following example mentioned in \cite{mai_infinite} shows.

\begin{example}\label{ex:not injective}
	For every $n \in \mathbb{N}$ fix some cube $C_n=[0,1]^n$. We glue them by identifying $C_n=[0,1]^n \times \{0\}^{n'-n} \subset [0,1]^{n'}=C_{n'}$ for $n \leq n'$. Then the resulting cube complex is not complete and hence not injective. For instance $x_n=(\frac{1}{2},\frac{1}{4}, \ldots , \frac{1}{2^n}) \in C_n$ defines a divergent Cauchy sequence. 
\end{example}


\section{Criterion for Injectivity}\label{sec:collapsible}



\subsection{Collapsible Cube Complexes}


\begin{definition}
	Let $\mathcal{C}$ be a cube complex. A subset $L$ of $|\mathcal{C}|$ is called \emph{cuboid} if it is connected and for all cubes $C \in \mathcal{C}$ the intersection $L \cap C$ is a face of $C$, given $L\cap C \neq \emptyset$.
	
	A connected subset $Y$ of $|\mathcal{C}|$ is called \emph{generalised cuboid} of $\mathcal{C}$ if for all cells $C=p_\lambda(I_\lambda^{n_\lambda}) \in \mathcal{C}$ either $Y \cap C= \emptyset$ or there are $0\leq s_i \leq t_i \leq 1$, for $i=1,\ldots,n_\lambda$, such that
	\begin{equation*}
		Y \cap C = \{ p_\lambda((y_1,\ldots,y_n)) : s_i\leq y_i \leq t_i, i=1,\ldots,n_\lambda\}.
	\end{equation*}
\end{definition}

\begin{definition}
	Let $\mathcal{C}$ be a cubical complex. $\mathcal{C}$ is called \emph{collapsible} if there is a sequence of subcomplexes $\mathcal{C}_0, \mathcal{C}_1, \ldots, \mathcal{C}_m$ of $\mathcal{C}$ and for each $k=0,1,\ldots ,m-1$ a family of nonempty connected subcomplexes $\mathcal{L}_{k,i}$ of $\mathcal{C}_k$, $i \in J_k$, such that $\mathcal{C}_0$ is one point, $\mathcal{C}_m=\mathcal{C}$ and $\mathcal{C}_{k+1}=\mathcal{C}_k \cup \bigcup_{i \in J_k} (\mathcal{L}_{k,i}\times \mathcal{I}_i)$, where
	\begin{equation*}
		\mathcal{I}_i=\{ \{0\} , \{1\}, [0,1] \} \text{ and } \mathcal{L}_{k,i}\times \mathcal{I}_i= \{ C\times D : C \in \mathcal{L}_{k,i} , D \in \mathcal{I}_i \},
	\end{equation*}
	for $k=0,1,\ldots,m-1$, and each $C \in \mathcal{L}_{k,i}$ is identified with $C\times \{0\}$. A collapsible cubical complex is \emph{regularly collapsible} if each $|\mathcal{L}_{k,i}|$ is a cuboid in $\mathcal{C}_k$.
	The sequence $\mathcal{C}_0, \mathcal{C}_1, \ldots, \mathcal{C}_m$ is then called a regular decomposition of $\mathcal{C}$.
\end{definition}

\begin{lem}\label{lem:hyperplanes in collapsible complex}
	Let $\mathcal{C}$ be a cube complex with a family of subcomplexes $\mathcal{L}_i, i\in J$. Then the set $\mathcal{H}'$ of hyperplanes in $\mathcal{C}'=\mathcal{C} \cup \bigcup_{i\in J} \mathcal{L}_i \times \mathcal{I}$ consists of extensions of hyperplanes in $\mathcal{C}$ and one hyperplane $\mathfrak{h}_i$ for each $i\in J$. Moreover the hyperplanes $\{ \mathfrak{h}_i : i\in J\}$ are pairwise disjoint and extremal. 
\end{lem}

\begin{proof}
	Observe that there are only two types of new edges $e' \in \mathcal{C}'\setminus\mathcal{C}$. Either it is of the form $e'=e\times \{1\}$ for some edge $e\in \mathcal{L}_i$ or  $e'=v\times [0,1]$ for some vertex $v\in \mathcal{L}_i$. Edges of the first type are always equivalent to an edge in $\mathcal{C}$, since $e\times \{1\} \Box e\times \{0\}$. Next we show that hyperplanes dual to some edge in $\mathcal{C}$ are extensions of hyperplanes in $\mathcal{C}$, i.e. that edges in $\mathcal{C}$ are equivalent in $\mathcal{C}'$ if and only if they are already equivalent inside $\mathcal{C}$. Assume that $e_0',\ldots,e_l'$ is a shortest sequence of edges between two edges $e_0',e_l' \in \mathcal{C}$ such that $e_k,e_{k+1}$ are opposite in some square. We claim that $e_k'\in \mathcal{C}$ for all $k$. If not there are is some maximal subsequence $e_{k_0}'=e_{k_0}\times \{1\},\ldots,e_{l_0}'=e_{l_0}\times\{1\}$ in $\mathcal{C}\setminus\mathcal{C}$, but then $e_0',\ldots,e_{k_0-1}'=e_{k_0}\times \{0\},e_{k_0+1}\times \{0\} , \ldots,e_{l_0-1}\times \{0\},e_{l_0}\times \{0\}=e_{l_0+1}',\ldots ,e_l'$ is strictly smaller.
	
	On the other hand, edges $v\times [0,1], v\in \mathcal{L}_i$ and $w\times [0,1],w\in \mathcal{L}_j$ of the second type are equivalent if and only if $i=j$. Indeed, if $i=j$, since $\mathcal{L}_i$ is connected, there is some discrete path $x_0,\ldots,x_l$ from $v$ to $w$ in $\mathcal{L}_i$ and $x_k\times [0,1]\Box x_{k+1}\times [0,1]$. Conversely squares containing $v\times [0,1]$ are always of the form $e\times [0,1], e\in \mathcal{L}_i$. Define $\mathfrak{h}_i$ as the hyperplane dual to all edges of the form $v\times [0,1], v\in \mathcal{L}_i$. Two such hyperplanes are disjoint since there is no square containing an edge dual to $\mathfrak{h}_i$ and one dual to $\mathfrak{h}_j$ at the same time.
\end{proof}

\begin{cor}\label{cor:finite width}
	A collapsible cube complex has finite width.
\end{cor}

\begin{proof}
	Let $\mathcal{C}$ be a cube complex with finite width and $\mathcal{C}'=\mathcal{C} \cup \bigcup_{i\in J} \mathcal{L}_i \times \mathcal{I}$ as before. We claim that $\mathcal{C}'$ also has finite width. Since all new hyperplanes are extremal, the corresponding halfspaces only occur at the ends of chains. Therefore $\operatorname{width} (\mathcal{C}') \leq \operatorname{width} (\mathcal{C}) + 2$.
\end{proof}

\begin{cor}\label{cor:colorable}
	The set of all hyperplanes in a collapsible cube complex is finitely colorable.
\end{cor}

\begin{proof}
	Let $\mathcal{C}$ be a cube complex with regular decomposition $\mathcal{C}_0, \mathcal{C}_1, \ldots, \mathcal{C}_m$. Lemma~\ref{lem:hyperplanes in collapsible complex} tells us that we can color hyperplanes added in one step with the same color. Hence the set of hyperplanes in $\mathcal{C}$ is $m$-colorable.
\end{proof}

\begin{example}\label{ex:not colorable}
By the simplex-graph construction described in \cite{chepoi-hagen} it is possible to construct a $2$-dimensional CAT(0) cube complex with diameter $4$ such that its hyperplanes are not finitely colorable. Given some graph $G$ the vertices of the simplex graph are given by the simplices of $G$ and two simplices are connected by some edge whenever they differ by one vertex. This is a median graph and therefore defines a CAT(0) cube complex $\mathcal{C}(G)$. Moreover the crossing graph $\Gamma(\mathcal{C}(G))$ is equal to $G$. If $G$ is triangle free then $\mathcal{C}(G)$ is $2$-dimensional and has diameter $\leq 4$. Therefore we are looking for some triangle free graph which is not finitely colorable. Starting with some triangle free graph, e.g. $M_1=(\{\{0\},\{1\}\},\{\{0,1\}\})$, and iterating the Mycielski construction we get a ascending sequence of triangle free graphs $M_1\subset M_2 \subset \ldots \subset M_n \subset \ldots$ with chromatic number $\chi(M_n)=n$. Then $M=\bigcup_{i=1}^\infty M_n$ is a triangle free graph with unbounded chromatic number. The corresponding CAT(0) cube complex $\mathcal{C}(M)$ has the desired properties. 
\end{example}


\subsection{Collapsible Cube Complexes are Injective}


\begin{lem} \label{lem:projection}
	Let $\mathcal{L}_i,i\in J,$ be a family of subcomplexes of a cubical complex $\mathcal{C}$, such that the $|\mathcal{L}_i|$'s are cuboids in $\mathcal{C}$. Define $\mathcal{C}'=\mathcal{C} \cup \bigcup _{i \in J} (\mathcal{L}_i \times \mathcal{I}_i)$. There is a projection map $p \colon |\mathcal{C}'| \to |\mathcal{C}|$ given by $p(x)=x$ for $x \in |\mathcal{C}|$ and $p(y,t)=y$ for $(y,t) \in |\mathcal{L}_i\times \mathcal{I}_i|$. If $X$ is a generalized cuboid of $\mathcal{C}'$ then
	\begin{enumerate}[(i)]
	\item If $X \cap |\mathcal{C}| = \emptyset$, then there is some $i \in J$ and $0<s_i \leq t_i \leq 1$ such that $X= p(X)\times [s_i,t_i] \subset |\mathcal{L}_i \times \mathcal{I}_i|$.
	\item If $X \cap |\mathcal{C}| \neq \emptyset$, then $X \cap |\mathcal{C}|=p(X)$ and for each $i$ either $X\cap |L_i| = \emptyset$ and $X\cap |L_i\times I_i|=\emptyset$ or otherwise there is some $0\leq t_i\leq 1$ such that $X\cap |\mathcal{L}_i\times \mathcal{I}_i|= (X\cap |\mathcal{L}_i|) \times [0,t_i]$.
	\item $p(X)$ is a generalized cuboid of $\mathcal{C}$.
	\end{enumerate}
\end{lem}

\begin{proof}
	First assume $X \cap |\mathcal{C}| = \emptyset$. Since $X$ is connected it is contained in some $|\mathcal{L}_i \times \mathcal{I}_i|$. Statement \textit{(i)} then follows as well as \textit{(ii)} from the following claim.
	\begin{cl}
		If $(y,t),(y',t') \in X \cap |\mathcal{L}_i\times \mathcal{I}_i|$ then $(y',t) \in X \cap |\mathcal{L}_i\times \mathcal{I}_i|$.
	\end{cl}
	Let $(y,t),(y',t') \in X \cap |\mathcal{L}_i\times \mathcal{I}_i|$. Then there is some sequence $(y,t)=(y_0,t_0), (y_1,t_1), \ldots,(y_n,t_n)=(y',t')$ in $X \cap |\mathcal{L}_i\times \mathcal{I}_i|$ such that for each $k=1,\ldots,n$ there is some cube $C_k$ with $(y_{k-1},t_{k-1}),(y_k,t_k) \in C_k$. We show inductively that $(y_k,t) \in X$. $(y_0,t) \in X$ by assumption. Thus let $(y_k,t) \in X$. We have $(y_k,t),(y_{k+1},t_{k+1}) \in X\cap C_{k+1} = \prod [a_m,b_m]$ and therefore $(y_{k+1},t) \in X\cap C_{k+1}$.

	It remains proving (iii). If $X \cap |\mathcal{C}| \neq \emptyset$ then $p(X)=X\cap \mathcal{C}$ and therefore is a generalized cuboid of $\mathcal{C}$. On the other hand, if $X \cap |\mathcal{C}| = \emptyset$, we have that $p(X)$ is a generalized cuboid in $\mathcal{L}_i$ and therefore, since $|\mathcal{L}_i|$ is a generalized cuboid itself, $p(X)$ is as well a generalized cuboid in $\mathcal{C}$.
\end{proof}

For $r \geq 0$ and subsets $Y \subset \mathcal{C}$ and $X \subset \mathcal{C}'$ we define
\begin{equation*}
	\bar{B}(Y,r)=\{y\in |\mathcal{C}| : d_\infty(y,Y) \leq r\}
\end{equation*}
\begin{equation*}
	\bar{B}'(X,r)=\{x\in |\mathcal{C}'| : d_\infty(x,X) \leq r\}
\end{equation*}

\begin{rmk}
	Let $X \subset |\mathcal{C}'|$ and $r\geq r_0\geq 0$. Then $\bar{B}'(X,r)=\bar{B}'(\bar{B}'(X,r_0),r-r_0)$.
\end{rmk}

\begin{lem}\label{lem:metric}
	The metric $d_\infty$ has the following properties:
	\begin{enumerate}[(i)]
	\item On $|\mathcal{L}_i \times \mathcal{I}_i|$ the metric fulfills $d_\infty ((y_1,t_1),(y_2,t_2))=\max \{ d_\infty(y_1,y_2), |t_1-t_2| \}$.
	\item For all $x,y \in |\mathcal{C}'|$ we have $d_\infty (p(x),p(y))\leq d_\infty (x,y)$.
	\end{enumerate}
\end{lem}

\begin{lem}\label{lem:balls}
	In the setting of Lemma~\ref{lem:projection} suppose that for every generalized cuboid $Y$ of $\mathcal{C}$ and every $s\geq 0$, $\bar{B}(Y,s)$ is a generalized cuboid of $\mathcal{C}$. Then for every generalized cuboid $X$ of $\mathcal{C}'$ and $r\geq 0$ we have that $\bar{B}'(X,r)$ is a generalized cuboid of $\mathcal{C}'$.
\end{lem}

\begin{proof}
	Let $X$ be a generalized cuboid of $\mathcal{C}'$ and $r\geq 0$.
	
	First assume that $|\mathcal{C}| \cap X \neq \emptyset$. Then $\bar{B}'(X,r) \cap |\mathcal{C}| = \bar{B}(X,r)$ since for $y \in |\mathcal{C}|$ and $x \in X$ we have 
	\begin{equation}
	d_\infty(y,p(x))=d_\infty(p(y),p(x))\leq d_\infty(y,x)
	\end{equation}
	and $p(x)\in X\cap |\mathcal{C}|$.
	
	For $i\in J$, if $X\cap |\mathcal{L}_i|\neq \emptyset$, using Lemma~\ref{lem:metric} (i), we easily get 
	\begin{equation}	
	\bar{B}'(X,r)\cap |\mathcal{L}_i\times \mathcal{I}_i| = (\bar{B}(X\cap|\mathcal{C}|,r)\cap |\mathcal{L}_i|)\times [0,t_i']
	\end{equation}
	where $t_i'=\min \{1,t_i+r\}$. On the other hand, if $X\cap |\mathcal{L}_i|= \emptyset$, define $r_0= d_\infty(X,|\mathcal{L}_i|)$. We distinguish two cases. If $r_0 < r$ we have $\bar{B}'(X,r)\cap |\mathcal{L}_i\times \mathcal{I}_i|=\emptyset$ and otherwise
	\begin{equation}
	\begin{split}
	\bar{B}'(X,r)\cap |\mathcal{L}_i\times \mathcal{I}_i| &= \bar{B}'(\bar{B}'(X,r_0),r-r_0)\cap |\mathcal{L}_i\times \mathcal{I}_i| \\
	&= (\bar{B}(X\cap|\mathcal{C}|,r)\cap |\mathcal{L}_i|)\times [0,t_i']
	\end{split}
	\end{equation}
	for $t_i'=\min \{ 1, r-r_0 \}$.
	
	We conclude that for $|\mathcal{C}| \cap X \neq \emptyset$ we have
	\begin{equation}
	\bar{B}'(X,r)=\bar{B}(X\cap |\mathcal{C}|,r) \cup \bigcup_{i\in J_X} (B(X\cap |\mathcal{C}|,r)\cap |\mathcal{L}_i|)\times [0,t_i'],
	\end{equation}		
	where $J_X=\{i\in J : d_\infty(X,|L_i|)\leq r\}$. It is connected and all components are generalized cuboids. For a cell $C$ of $\mathcal{C}$ we have $C\cap \bar{B}'(X,r)=C\cap \bar{B}(X\cap |\mathcal{C}|,r)$ and for a cell $C$ of $\mathcal{L}_i\times \mathcal{I}_i$, $C\cap \bar{B}'(X,r)=C\cap ((\bar{B}(X\cap |\mathcal{C}|,r)\cap |\mathcal{L}_i|)\times [0,t_i'])$. Therefore $\bar{B}'(X,r)$ is a generalized cuboid as well.
	
	Let us now consider the case $|\mathcal{C}| \cap X = \emptyset$. Then $X \subset |\mathcal{L}_i\times \mathcal{I}_i|$ for some $i\in J$ and
	\begin{equation}
		X=p(X)\times [s_i,t_i]
	\end{equation}
	If $r \leq s_i$ we have $\bar{B}'(X,r)=\bar{B}(p(X)\cap |\mathcal{L}_i|,r)\times [s_i-r,t_i']$ with $t_i'=\min \{t_i+r,1\}$ which is a generalized cuboid. If $r> s_i$ observe that $\bar{B}'(X,s_i)$ is a generalized cuboid with $\bar{B}'(X,s_i)\cap |\mathcal{C}| \neq \emptyset$ and therefore $\bar{B}'(X,r)=\bar{B}'(\bar{B}'(X,s_i),r-s_i)$ is a generalized cuboid by the first case.
\end{proof}

\begin{definition}
	A cube complex $\mathcal{C}$ has \emph{property (P)} if the following holds. Let $\{X_\alpha : \alpha \in A\}$ be a collection of generalized cuboids of $\mathcal{C}$ with $X_\alpha \cap X_\beta \neq \emptyset$ for all $\alpha, \beta \in A$. Then $\bigcap_{\alpha\in A}X_\alpha \neq \emptyset$.
\end{definition}

\begin{lem}\label{lem:property (P)}
	Let $\mathcal{C},\mathcal{C}'$ be as in Lemma~\ref{lem:projection}. If $\mathcal{C}$ has property (P) then $\mathcal{C}'$ has property (P) as well.
\end{lem}

\begin{proof}
	Let $\{X_\alpha : \alpha \in A\}$ be a collection of generalized cuboids of $\mathcal{C}'$ with $X_\alpha \cap X_\beta \neq \emptyset$ for all $\alpha, \beta \in A$. Then since $\mathcal{C}$ has property (P) we know that $\bigcap_{\alpha \in A} p(X_\alpha) \neq \emptyset$.
		
	If $X_\alpha \cap |\mathcal{C}| \neq \emptyset$ for all $\alpha \in A$ we have
	\begin{equation}
		(\bigcap_{\alpha \in A} X_\alpha ) \cap |\mathcal{C}|=\bigcap_{\alpha \in A} (X_\alpha\cap |\mathcal{C}|)=\bigcap_{\alpha \in A} p(X_\alpha) \neq \emptyset
	\end{equation}
	and therefore $\bigcap_{\alpha \in A} X_\alpha\neq\emptyset$.
	
	If $X_{\alpha_0} \cap |\mathcal{C}| = \emptyset$ for some $\alpha_0 \in A$ then $X_{\alpha_0}=p(X_{\alpha_0})\times[s_{\alpha_0},t_{\alpha_0}] \subset |\mathcal{L}_i\times \mathcal{I}_i |$ for some $i\in J$ and $0 < s_{\alpha_0} \leq t_{\alpha_0} \leq 1$. Especially we have for each $\alpha\in A$ that $X_\alpha \cap |\mathcal{L}_i\times \mathcal{I}_i|\neq \emptyset$ and thus there are $0 \leq s_\alpha \leq t_\alpha \leq 1$ such that
	\begin{equation}
		X_\alpha \cap |\mathcal{L}_i\times \mathcal{I}_i|= (p(X_\alpha)\cap |\mathcal{L}_i|) \times[s_\alpha,t_\alpha].
	\end{equation}
	Define $s=\sup \{s_\alpha : \alpha \in A\}$ and $t=\inf\{t_\alpha : \alpha\in A\}$. Then $s\leq t$ since otherwise there are $\alpha_1,\alpha_2 \in A$ with $s_{\alpha_1} > t_{\alpha_2}$. But then $X_{\alpha_1} \subset |\mathcal{L}_i \times \mathcal{I}_i|$ and $X_{\alpha_1}\cap X_{\alpha_2}=\emptyset$. Hence we have
	\begin{equation}
		\bigcap_{\alpha \in A} X_\alpha =(\bigcap_{\alpha \in A} X_\alpha) \cap |\mathcal{L}_i \times \mathcal{I}_i| = (\bigcap_{\alpha \in A} p(X_\alpha))\times [s,t] \neq \emptyset.
	\end{equation}	
\end{proof}

\begin{proof}[Proof of Theorem~\ref{thm:injective}.]
	Let $\mathcal{C}_0, \mathcal{C}_1, \ldots, \mathcal{C}_n$ be a regular decomposition of $\mathcal{C}$ into subcomplexes. First observe that by Lemma~\ref{lem:balls} in each $\mathcal{C}_k$ closed balls are generalized cuboids. Furthermore all subcomplexes $\mathcal{C}_k$ have property (P) by Lemma~\ref{lem:property (P)}. But this tells us that $(|\mathcal{C}|,d_\infty)$ is hyperconvex and therefore injective.
\end{proof}

In view of Theorem~\ref{thm:CAT(0)} we conclude that collapsible cube complexes are CAT(0). This completes the first implication in Theorem~\ref{thm:collapsible} since collapsible cube complexes have finitely colorable hyperplanes and finite width (see Corollaries~\ref{cor:finite width} \& \ref{cor:colorable}).

\begin{cor}\label{cor:CAT(0)}
	Let $\mathcal{C}$ be a regularly collapsible cube complex. Then $(|\mathcal{C}|,d_2)$ is a CAT(0) space.
\end{cor}


\subsection{Collapsible CAT(0) Cube Complexes}


\begin{lem}\label{lem:simple collaps}
	Let $\mathfrak{h}$ be a hyperplane in a CAT(0) cube complex $\mathcal{C}$ which splits the underlying graph $G$ into the halfspaces $H$ and $H'$. If $H$ is minimal with respect to inclusion then there is an isomorphism $G \approx (G \setminus H) \cup (p'(H) \times  \{0,1\})$, where $p' \colon G \to H'$ is the nearest point projection. Its linear extension then yields an isomorphism
	\begin{equation}
		| \mathcal{C} | \approx | \mathcal{C} \setminus H | \cup (| p'(H) | \times [0,1]).
	\end{equation}		
	Moreover $| p'(H) |$ is a cuboid in $|\mathcal{C} \setminus H |$.
\end{lem}

\begin{proof}
	Denote by $H'$ the complement of $H$. From Lemma~\ref{lem:convex subgraphs} we know that there are gate functions $p\colon G \to H$ and $p' \colon G \to H'$. First we show that $p(H')=H$. Indeed if there is $x \in H \setminus p(H')$ we find some hyperplane $\mathfrak{h}_0$ separating $x$ from $p(H')$. Label the halfspace containing $x$ with $H_0$. But then $H_0 \cap H'$ is mapped into $H_0$ and therefore must be empty. We have $H_0 \subset H$, a contradiction to the minimality of $H$. 

	By Lemma~\ref{lem:gates of gates} it therefore follows that $p' \colon H \to p'(H)$ is an isomorphism between the median graphs $H$ and $p'(H)$. Since every point in $H$ has distance 1 to its image we get that $p'(H) \cup H$ is isomorphic to $p'(H) \times \{0, 1\}$ which induces an isomorphism
	\begin{equation}
		G \approx (G \setminus H) \cup p'(H) \times \{0, 1\}
	\end{equation}
Its linear extension then gives us the desired isomorphism 
	\begin{equation}
		|\mathcal{C}| \approx |\mathcal{C} \setminus H | \cup (|p'(H)|\times [0,1]).
	\end{equation}
	As we have proven in Lemma~\ref{lem:convex image} $p(H)$ is a convex subset of $G$ and therefore its linear extension is a cuboid in $| \mathcal{C} \setminus H |$.
\end{proof}

\begin{cor}\label{cor:multicollaps}
	Let $\mathcal{C}$ be a CAT(0) cube complex and $\mathcal{H}$ a familiy of pairwise disjoint extremal hyperplanes. Denote by $H(\mathfrak{h})$ the minimal halfspace bounded by $\mathfrak{h}$ (if both halfspaces are minimal, fix a specific choice). Moreover let $p_\mathfrak{h} \colon \mathcal{C} \to \mathcal{C}\setminus H(\mathfrak{h})$ be the projection onto the complement of $H(\mathfrak{h})$ and $\mathcal{C}'= \mathcal{C}\setminus \bigcup_{\mathfrak{h}\in \mathcal{H}} H(\mathfrak{h})$. Then there is an isomorphism
	\begin{equation}
		|\mathcal{C}| \approx |\mathcal{C}'| \cup \bigcup_{\mathfrak{h}\in \mathcal{H}} \left(| p_\mathfrak{h}(H(\mathfrak{h}))| \times [0,1] \right).
	\end{equation}
\end{cor}

\begin{proof}
	Let $p\colon \mathcal{C} \to H$ be a projection map onto some halfspace $H$. We claim that its image is disjoint to any minimal halfspace $H'$ with $\partial H \cap \partial H' = \emptyset$. Indeed for $x \in p(H)$ there is some edge $e$ that is orthogonal to $H$ and contains $x$. Thus if $x$ is contained in the halfspace $H'$, we also have $e \in H'$ and hence $\partial H \subset H'$ since $\partial H \cap \partial H' = \emptyset$. But then $H$ or $H^C$ is completely contained in $H'$, contradicting minimality.
	Together with the observation that a minimal halfspace is the image of a projection map in the previous proof, it follows that $H(\mathfrak{h}) \cap H(\mathfrak{h'})= \emptyset$ and $p_\mathfrak{h} (H(\mathfrak{h})) \cap H(\mathfrak{h'})= \emptyset$ for $\mathfrak{h}\neq \mathfrak{h'} \in \mathcal{H}$. Therefore we have $p_\mathfrak{h}(H(\mathfrak{h})) \subset \mathcal{C}'$. By Lemma~\ref{lem:simple collaps} there is some isomorphism $\Phi_\mathfrak{h} \colon \mathcal{C} \to (\mathcal{C}\setminus H(\mathfrak{h}) \cup (p_\mathfrak{h}(H(\mathfrak{h}))\times\{0,1\})$. We combine them to the following canonical isomorphism $\Phi \colon 
\mathcal{C} \to \mathcal{C}' \cup \bigcup_{\mathfrak{h}\in \mathcal{H}} \left( p_\mathfrak{h}(H(\mathfrak{h})) \times \{0,1\} \right)$ with
\begin{equation}
	\Phi(x) = \begin{cases}
				x, &x\in \mathcal{C}', \\
				\Phi_\mathfrak{h}(x), &x \in H(\mathfrak{h}).
			\end{cases}		
\end{equation}
\end{proof}

We are now in a position to prove the second implication in Theorem~\ref{thm:collapsible}.

\begin{thm}\label{thm:collapsible CAT(0) complexes}
	Every CAT(0) cube complex $\mathcal{C}$ with finite width and finitely colorable hyperplanes is regularly collapsible.
\end{thm}

\begin{proof}
	We prove this statement by induction on the width of $\mathcal{C}$. A $CAT(0)$ cube complex with width $0$ is a point and hence collapsible.
	So let $\mathcal{C}$ be a CAT(0) cube complex of width $w$ and coloring $\Phi \colon \mathcal{H} \to \{1, \ldots, n\}$. Consider the set $\mathcal{H}'$ of hyperplanes $\mathfrak{h}$ which bound a minimal halfspace $H(\mathfrak{h})$. Define $\mathcal{H}_1 = \mathcal{H}' \cap \Phi^{-1} (1)$ and $\mathcal{C}_1= \mathcal{C} \setminus \bigcup_{\mathfrak{h}\in \mathcal{H}_1} H(\mathfrak{h})$. Then 
	\begin{equation}
	|\mathcal{C}| \approx |\mathcal{C}_1| \cup \bigcup_{\mathfrak{h}\in \mathcal{H}_1} |p(H(\mathfrak{h}))| \times [0,1]
	\end{equation}		
	according to Corollary~\ref{cor:multicollaps}. This means that $\mathcal{C}$ collapses into $\mathcal{C}_1$.
	
	Observe that the hyperplanes in $\mathcal{C}_1$ are exactly the restriction of the hyperplanes in $\mathcal{C}$ which are not contained in $\mathcal{H}_1$. Therefore we can repeat the previous step starting with $\mathcal{C}_1$ and $\mathcal{H}_2= \{ \mathfrak{h} \cap \mathcal{C}_1 : \mathfrak{h} \in \mathcal{H}' \cap \Phi^{-1} (2) \}$. We get some CAT(0) cube complex $\mathcal{C}_2 = \mathcal{C}_1 \setminus \bigcup_{\mathfrak{h}\in \mathcal{H}_2} H(\mathfrak{h})$ which is the result of performing a multicollaps on $\mathcal{C}_1$.
	
	Proceeding this way we finally arrive by performing finitely many multicollapses at the CAT(0) cube complex $\mathcal{C}_n = \mathcal{C} \setminus \bigcup_{\mathfrak{h}\in \mathcal{H}'} H(\mathfrak{h})$. For completing the induction step it remains to show that $\mathcal{C}_n$ has strictly smaller width than $\mathcal{C}$. But this is true since every chain of halfspaces in $\mathcal{C}_n$ is contained in some chain in $\mathcal{C}$ and there every chain ends with some halfspace bounded by some hyperplane in $\mathcal{H}'$.
\end{proof}

\textbf{Acknowledgments.} I would like to thank Prof. Dr. Urs Lang for reading this and earlier versions of the paper and making helpful suggestions. This work was supported by the Swiss National Science Foundation.
	
\newpage

\nocite{*}

\bibliographystyle{amsplain}
\bibliography{mybib}

\small
\noindent\textsc{Department of Mathematics, ETH Z\"urich, 8092 Z\"urich, Switzerland}\\
\textit{E-mail address}: benjamin.miesch@math.ethz.ch

\end{document}